\documentclass[reqno,12pt]{article}

\usepackage{epsfig,amsfonts,latexsym,amsmath,amssymb,amsthm,amstext,euscript}

\begin{document}

\newtheorem{lemma}{Lemma}[section]
\newtheorem{theo}[lemma]{Theorem}
\newtheorem{coro}[lemma]{Corollary}
\newtheorem{rema}[lemma]{Remark}
\newtheorem{propos}[lemma]{Proposition}
\newtheorem{problem}[lemma]{Problem}

\def\bdem{\begin{proof}}
\def\edem{\end{proof}}
\def\bequ{\begin{equation}}
\def\eequ{\end{equation}}

\hyphenation{cha-rac-te-ri-za-tion} \hyphenation{in-ver-ti-ble}
\hyphenation{res-pec-ti-ve-ly}  \hyphenation{ve-ri-fy}
\hyphenation{na-tu-ral}  \hyphenation{co-ro-lla-ry}
\hyphenation{en-cou-ra-ge-ment}  \hyphenation{o-ther-wise}
\hyphenation{re-gu-la-ri-zed}  \hyphenation{sa-tis-fies}
\hyphenation{pa-ra-me-ters} \hyphenation{subs-tan-tia-lly}
\hyphenation{po-si-ti-ve}  \hyphenation{pro-ducts}
\hyphenation{pro-blem}  \hyphenation{ge-ne-ral}
\hyphenation{pa-ra-me-ter}

\renewcommand{\thesection}{\arabic{section}}
\renewcommand{\theequation}{\thesection.\arabic{equation}}
\newcommand{\equnew}{\setcounter{equation}{0}}

\newcommand{\chu}{\mbox{$ M ( H^{\infty} ) $}}
\newcommand{\rr}{\mbox{$    \rightarrow   $}}
\newcommand{\papa}{H^{\infty}}
\newcommand{\disc}{ {\Bbb D} }
\newcommand{\ov}{ \overline }
\newcommand{\om}{\omega}
\newcommand{\Om}{\Omega}
\newcommand{\dist}{\mbox{dist}}
\newcommand{\diam}{\mbox{diam}}
\newcommand{\laba}{\label}
\newcommand{\peso}[1]{ \quad \text{ \rm  #1 } \quad }
\newcommand{\papae}{H^{\infty}_{E}}
\newcommand{\bsr}{\mbox{Bsr}\,}
\newcommand{\dsr}{\mbox{dsr}\,}
\newcommand{\noi}{\noindent}
\newcommand{\supp}{\mbox{supp}\,}
\newcommand{\fa}{\mathcal{F}_a}
\newcommand{\fr}{\mathcal{F}_r}
\newcommand{\length}{\mbox{length}}
\newcommand{\eiti}{e^{i\theta}}
\newcommand{\meiti}{e^{-i\theta}}
\renewcommand{\Re}{\mbox{Re}\,}
\renewcommand{\Im}{\mbox{Im}\,}
\newcommand{\Log}{\mbox{Log}\,}
\newcommand{\Arg}{\mbox{Arg}\,}
\newcommand{\la}{\langle}
\newcommand{\ra}{\rangle}
\newcommand{\papai}{(H^\infty)^{-1}}
\newcommand{\Gammao}{\Gamma^o}
\newcommand{\eps}{\varepsilon}
\newcommand{\hol}{\mathcal{H}}
\newcommand{\bmo}{\mbox{BMO}}
\newcommand{\inn}{\mathfrak{I}}
\newcommand{\cn}{\mbox{CN}}
\newcommand{\inni}{\mathfrak{I}^{*}}
\newcommand{\cni}{CN^{*}}
\newcommand{\interi}{\mbox{int}\,}
\newcommand{\CNBP}{CNBP}
\newcommand{\n}{\mathfrak{N}}
\renewcommand{\O}{\mathcal{O}}
\newcommand{\ca}{\mathcal{C}}
\newcommand{\Int}{Int\,}


\title{Paths of inner-related functions}

\author{Artur Nicolau and Daniel Su\'{a}rez}
\date{December 17, 2010}

\maketitle
\begin{quotation}
\noindent
\mbox{ } \hfill      {\sc Abstract}    \hfill \mbox{ } \\
\footnotetext{2010 Mathematics
Subject Classification: primary 30H05, secondary 46J15.
Key words: inner functions, Carleson-Newman Blaschke products, connected components. } \hfill \mbox{ }\\
{\small \noindent
We characterize the connected components of the subset $\cni$ of $\papa$ formed by the products $bh$, where
$b$ is Carleson-Newman Blaschke product and $h\in\papa$ is an invertible function. We use this result to show that,
except for finite Blaschke products, no inner function in the little Bloch space is in the closure of one
of these components.
Our main result says that every inner function can be connected with an element of $\cni$ within the set
of products $uh$, where $u$ is inner and $h$  is invertible. We also study some of these issues in the context of
Douglas algebras.
}  \\
\end{quotation}

\setcounter{section}{0}

\section{Introduction}

Let $H^\infty$ be the Banach algebra of bounded analytic functions
in the unit disk $\disc$ with the norm $\|f\|_\infty=\sup_{z\in
\disc} |f(z)|$.  A function in $H^\infty$ is called inner if it has
radial limits of modulus one at almost every point of the unit
circle $\partial \disc$. A Blaschke product is an inner function of
the form
\begin{equation*}
b(z)=\lambda z^m \prod_{n=1}^\infty \frac{\bar z_n}{|z_n|}\, \frac{z_n -z}{1-\bar z_n z}\,,
\end{equation*}
where $m$ is a nonnegative integer, $\lambda\in \partial \disc$ and $\{z_n\}$ is a sequence of
points in $\disc \setminus \{0\}$ satisfying the Blaschke condition $\sum_n (1-|z_n|)<\infty$.
If $\lambda=1$, we say that $b$ is normalized.
Given a Blaschke product $b$, we denote by $Z(b)=\{z_n\}$ the sequence of its zeros
repeated according to their multiplicity. A classical result by O. Frostman tells us that for any inner function $u$, there exists
an exceptional set $E=E(u)\subset \disc$ of logarithmic capacity
zero such that the Mobius shift $(u-\alpha)/(1-\bar \alpha u)$ is a
Blaschke product for any $\alpha \in \disc \setminus E$ (see
\cite{fro} or \cite[pp.$\,$79]{gar}). Hence, any inner function
can be uniformly approximated by Blaschke products. The set of
invertible functions in $\papa$ will be denoted by $(H^\infty)^{-1}$
and it consists of those functions $h\in H^\infty$ satisfying $\inf
|h(z)|>0$, where the infimum is taken over all points $z\in \disc$.
Let $\inni$ be the open set in $\papa$ of functions of the form $f=uh$, where $u$ is an inner function and
$h\in (H^\infty)^{-1}$.
Equivalently, $\inni$ consists of those
functions in $ \papa$ whose non-tangential limits on the unit circle
are bounded below away from zero. A result by Laroco asserts that $\inni$ is dense in $H^\infty$ (see \cite{lar}).

A sequence of points $\{z_n\}$ of the unit disk is called an
interpolating sequence if for any bounded sequence of complex values
$\{w_n\}$, there exists a function $f\in H^\infty$ with
$f(z_n)=w_n$, $n=1,2,\dots$ A celebrated result by Carleson
(\cite{ca.i} or \cite[pp.$\,$287]{gar}) asserts that $\{z_n\}$ is
an interpolating sequence if and only if
\begin{equation*}
\inf_n (1-|z_n|^2) |b'(z_n)|>0,
\end{equation*}
where $b$ is the Blaschke product with zeros $\{z_n\}$. Although
interpolating Blaschke products comprise a small subset of all inner
functions, they play a central role in the theory of the algebra
$H^\infty$. See for instance the last three chapters of
\cite{gar}.
Marshall proved that finite linear combinations of Blaschke products are dense in
$H^\infty$ (see \cite{mar}). Later, this result was sharpened in
\cite{g-n} by showing that one can use interpolating
Blaschke products. However, the following problem remains open.

\begin{problem}\laba{prob1}
For any inner function $u$ and $\varepsilon >0$, does there exist an
interpolating Blaschke product $b$ such that $\| u-b\|_\infty
<\varepsilon$?
\end{problem}

This question was posed in \cite{jones} and  \cite[pp.$\,$420]{gar}, and it is one of the
most important open problems in the area. The following weaker
version of Problem~\ref{prob1} is also open.

\begin{problem}\laba{prob2}
For any inner function $u$ and $\varepsilon >0$, does there exist an
interpolating Blaschke product $b$ and an invertible function $h\in
H^\infty$ such that $\|u-bh\|_\infty<\varepsilon$?
\end{problem}

This is really a question of approximation in BMO. Recall that a
function $f\in L^1 (\partial \disc)$ is in the space $BMO$ if
\begin{equation*}
\|f\|_*= \sup\, \frac{1}{|I|} \int_I |f-f_I|<\infty,
\end{equation*}
where the supremum is taken over all arcs $I\subset \partial \disc$
of the unit circle and $f_I=|I|^{-1}\int_I f $
 is the mean of $f$ over the arc $I$. A classical result by Fefferman and Stein says that a function $f \in L^1 (\partial \disc)$
 is in BMO if and only if $f$ can be written as $f=r+\tilde s$, where $r, s \in L^\infty(\partial \disc)$.
 Here $\tilde s$ means the harmonic conjugate  of $s$. Moreover, $\|f\|_*$ is comparable to $\|f\|_{BMO} =
 \inf  \{ \|r\|_\infty + \|s\|_\infty \}$, where the infimum is taken over all possible decompositions
 $f=r+\tilde s +c$, where $c$ is a constant. It is easy to see that
 Problem~\ref{prob2} has a positive answer if and only if
 for any inner function $u$ and any $\varepsilon >0$,
 there exists an interpolating Blaschke product $b$ such that a suitable branch $\operatorname{Arg} (u/b)$
  of the argument of the function $u(\xi)/b(\xi)$, $\xi \in \partial \disc$,  satisfies
\begin{equation*}
\|\operatorname{Arg} (u/b)\|_{BMO} \leq \varepsilon.
\end{equation*}
In other words, Problem~\ref{prob2} is the BMO-version of
Problem~\ref{prob1}.

In connection with the theory of Toeplitz operators on Hardy spaces
and the existence of unconditional basis of reproducing kernels
in model spaces, Nikolskii proposed the following weak version of
Problem~\ref{prob2}, which is also still open
\cite[pp.$\,$210]{nik} (see also \cite[pp.$\,$91--93]{Seip}).

\begin{problem}\laba{prob3}
For any inner function $u$, is there an interpolating Blaschke
product $b$ such that
\begin{equation*}
\operatorname{dist} (u\bar b, H^\infty) <1
\,\ \mbox{ and }\ \,
\operatorname{dist} (b \bar u , H^\infty) <1?
\end{equation*}
\end{problem}
\noi An equivalent formulation is to ask whether there is an interpolating Blaschke product $b$ and
$h\in (\papa)^{-1}$  such that $\| u - b h \|_\infty <1$ (see \cite[pp.$\,$220]{nik}).

A Blaschke product $b$ is called a Carleson-Newman Blaschke product
if it can be decomposed as a finite product of interpolating
Blaschke products, or equivalently, if
$$\mu_b = \sum_{ b(z)=0} (1-|z|) \delta_z $$
is a Carleson measure. Here $\delta_z$ denotes the Dirac measure
at the point $z$. Recall that a complex-valued measure $\mu$ in the
unit disk is called a Carleson measure if there exists a constant
$C>0$ such that
$$
\int_{\disc} |f(z)| d |\mu| (z) \leq C \|f \|_1
$$
for any function $f$ in the Hardy space $H^1$. The infimum of such
constants $C$ is denoted by $\|\mu \|_c$.  It is well known that any
Carleson-Newmann Blaschke product can be approximated uniformly by
 interpolating Blaschke products (see \cite{m-s}).
Thus, one can interchange interpolating and Carleson-Newman Blaschke products in the questions above, as well as in the
rest of the paper.
Let $\cni$ be the open set in $H^\infty$ of functions of the form $f=b h$ where $b$ is a Carleson-Newman Blaschke product
and $h\in (H^\infty)^{-1}$.
Equivalently, $\cni$ consists of those functions  $ f \in \papa$ for which there exists a number $0<r<1$, such that for any $z
\in \disc$ one has $\sup \{|f(w)| : |w-z| < r(1-|z|)\} >0$.

The main purpose of this paper is to study the connected components of $\inni$ and
$\cni$. This shall lead us to answer natural analogues of Problems~\ref{prob2} and \ref{prob3}, as well as to consider a
number of related questions. The components of the set of inner functions has been considered by Herrero in
\cite{her} and by Nestoridis in \cite{nes1} and \cite{nes2}.

A continuous version of Problem~\ref{prob2} can be stated as
follows: given an inner function $u$, does there exist a path in $\cni$ (except for the final point) which ends at $u$?
Or more basically, is any inner function in the closure of a connected component of $\cni$? We prove that the
answer to both questions is negative.
Recall that an analytic function $f$ on the unit
disk is in the little Bloch space if
\begin{equation*}
\lim_{|z|\to 1} (1-|z|^2)\, |f'(z)|=0.
\end{equation*}
Blaschke products with finitely many zeros are in the little Bloch
space but it also contains many other inner functions.
We prove in Theorem \ref{litbloch} that any inner function in the little Bloch space which is not
a finite Blaschke product does not belong to the closure of a single component of $\cni$.
The proof uses a description of the connected components of $\cni$ given in Theorem \ref{cra}. Roughly speaking, the component of a function $f\in\cni$ is described
in terms of the zeros of $f$ and the Fefferman-Stein decomposition of the Cauchy integral of a path measure associated to those zeros.

In contrast to the situation described above for functions in the little Bloch space, any component of $\inni$ contains an element of $\cni$.
This is stated in the main result of the paper, Theorem \ref{components}.
It can be understood as a positive answer to a weaker
version of Problem~\ref{prob3}. Actually, if $\|u-bh\|_\infty<1$, the
segment $\gamma (t)=u+  t(bh-u)$, where $t\in [0,1]$, is contained in $\inni$ and joins
$\gamma (0)=u$ with $\gamma (1)=bh$. Our proof shows that there exists a universal
constant $N$ such that for any inner function $u$ there is a
polygonal $\gamma \colon [0,1] \to \inni$ formed by at most $N$
segments so that $\gamma (0)=u$ and $\gamma (1) \in \cni$.
The proof is constructive and it is the deepest part of the paper.
It uses a Carleson contour decomposition and a discretization of harmonic measures in its interior.
This provides a path in $L^\infty(\partial\disc)$ which can be lifted to $\inni$.

We will abbreviate Carleson-Newman Blaschke product by \CNBP.
The paper is organized as follows. Section 2 contains the description of the components of $\cni$ and the result on inner functions in the little Bloch space.
Section 3 is devoted to the main result of the paper.
Finally, in Section 4 we relate the previous results to Douglas algebras, present an example that illustrates
the fact that two arbitrary functions in $\inni$ or $\cni$ can be multiplied by a \CNBP\ into a single component,
and mention some open problems.\\

\section{The connected components of $\cni$}

Let $\mu$ be a finite complex measure in the complex plane. Let
$C_{\varepsilon} (\mu)$ be its truncated Cauchy integral defined in the unit circle as
$$
\ca_{\varepsilon} (\mu) (\eiti)= \int_{|z - \eiti|> \varepsilon} \frac{ d\mu(z)}{\eiti-z} \, .
$$
It was shown by Mattila and Melnikov that the Cauchy integral
defined as  $\ca(\mu) (\eiti)  = \lim_{\varepsilon \to 0} C_{\varepsilon} (\mu)(\eiti)$
exists at almost every point $\eiti \in \partial \disc$ (see \cite{m-m}). This was a consequence of the following
weak-$L^1$ estimate: there is a universal constant $C$ such that for any $\lambda >0$,
$|\{\eiti \in \partial \disc: \ca^* (\mu)(\eiti) > \lambda \} | < C \lambda^{-1} \| \mu \|$. Here $\ca^* (\mu)$ denotes the
maximal Cauchy transform defined as
$$
\ca^* (\mu)(\eiti)= \sup_{\varepsilon > 0}\left|\int_{ |z - \eiti|>\varepsilon} \frac{ d\mu(z)}{\eiti-z} \right|  .
$$
We start with a well-known result on Cauchy's integrals of Carleson
measures which will be used later. For $0<p< \infty$ let $H^p$ be
the Hardy space of analytic functions $f$ in the unit disk for which
$$
\|f \|_p^p =\sup_{r<1} \int_{0}^{2 \pi} |f(r e ^{i \theta})|^p d
\theta < \infty
$$
and $H^p_0$ be the subspace of those $f \in H^p$ with $f(0)=0$.

\begin{lemma}\laba{cauchy-adj}
Let $\mu$ be a complex-valued Carleson measure on $\disc$ and for
$0<r<1$, let $\mu_r$ be its restriction to the disk $r \disc$.
Then
\begin{enumerate}
\item[{\em (1)}] $\ca(\mu_r) $ converges in $L^2$-norm to $C(\mu)$ as $r \to
1$.
\item[{\em (2)}]
$\ov{\ca(\mu)} \in H^2_0$ and $\| \ca(\mu) \|_2 \leq     \| \mu
\|_c^{1/2} \, |\mu|(\disc)^{1/2}$,
\item[{\em (3)}]
$\ov{\ca(\mu)} \in BMO$ and $\| \ca(\mu) \|_{BMO} \leq C \| \mu \|_c$,
where $C$ is an absolute constant.
\end{enumerate}
\end{lemma}
\begin{proof}
Given two functions $f, g \in L^2 (\partial \disc) $, let $\la f,g \ra$ denote their scalar product in
$L^2 (\partial \disc)$. It is
obvious that $h_r(\eiti) = \ov{\eiti \ca(\mu_r)(\eiti)}$ can be
extended to a function in $\papa$, and for $f\in H^2$ we have
$$
\la f,h_r \ra = \int_0^{2\pi} f(\eiti) \int_\disc
\frac{1}{1-z\meiti} \, d\mu_r(z) \, \frac{d\theta}{2\pi}
= \int_\disc f(z) \, d\mu_r(z).
$$
By the Cauchy-Schwarz inequality
$$
|\la f,h_r \ra | \leq \left( \int_\disc |f(z)|^2 \, d|\mu_r|(z)
\right)^{1/2}  |\mu_r|(\disc)^{1/2} \leq            \| \mu_r
\|_c^{1/2} \, \|f\|_2 \, |\mu_r|(\disc)^{1/2}.
$$
Since $\| \mu_r \|_c \leq \| \mu \|_c$, we get $\| h_r\|_2 \leq  \|
\mu \|_c^{1/2} \, |\mu_r|(\disc)^{1/2}$. Similarly,
$$\|h_r - h_s \|_2  \leq    \| \mu \|_c^{1/2} \, |\mu_r - \mu_s |(\disc)^{1/2},$$
from which $h_r$ converges in $L^2 (\partial \disc)$ to a function
$h$ when $r\rr 1$, with $\| h\|_2 \leq     \| \mu \|_c^{1/2} \,
|\mu|(\disc)^{1/2}$. Observe that $|\ca_{1-r} (\mu) (\eiti) - \ca(\mu_r)
(\eiti)| \leq \ca^* ( \mu -\mu_r) (\eiti)$.
Then the weak-$L^1$ estimate tells us that $\ca(\mu_r)$ converges in
measure to $\ca(\mu)$ as $r$ tends to $1$. Thus $h (\eiti) = \ov{\eiti
\ca(\mu)(\eiti)}$ at almost every point $\eiti  \in \partial \disc$
and (1) and (2) follow.

Let $f\in \papa$ and write $f(\eiti)-f(0)= \eiti g(\eiti)$, with
$g\in \papa$. Then
$$
|\la f, \ov{\ca (\mu)} \ra | =  |\la f-f(0), \ov{\ca(\mu)} \ra | = \left|
\int_\disc \int_0^{2\pi} \frac{g(\eiti)}{1-z\meiti} \,
\frac{d\theta}{2\pi} \, d\mu(z) \right| = \left|  \int_\disc g(z) \,
d\mu(z) \right|  ,
$$
which is bounded by $\|\mu\|_c \| f-f(0)\|_1
\leq 2 \|\mu\|_c  \|f\|_1$. Since  $\papa$ is dense in $H^1$,
$\ov{\ca(\mu)}$ induces a bounded linear functional on $H^1$ with norm
bounded by $2\|\mu\|_c$. This means that $\ov{\ca(\mu)}\in BMO$ with
$\| \ov{\ca(\mu)}\|_{BMO} \leq C \|\mu \|_c$.
\end{proof}

\begin{rema}\laba{rema1}
{\em
It is clear that the truncated measure $\mu_r$ in Lemma
\ref{cauchy-adj} can be replaced by $\chi_{E_s} \mu$, where $E_s$,
with $0<s<1$, is any continuum of increasing compact sets in
$\disc$ such that for any $0<r<1$ there is $s$ with $r\disc
\subset E_s$.
}
\end{rema}
\vspace{2mm}
\noi Let $\varphi_z(w)= (z-w)/(1- \ov{w}z)$ be the involution on $\disc$ that interchanges $0$ and $z$.
The pseudohyperbolic and hyperbolic distance between $z$ and $w$ in $\disc$ are respectively defined
by $\rho(z,w)= |\varphi_z(w)|$ and
$$
\beta(z, w) = \log \frac{1+ \rho(z,w)}{1- \rho(z,w)} .
$$
Let $b, b^\ast$ be two \CNBP$\ $  and
let $\{ z_k \}$, $\{ z^\ast_k \}$ be their zero sequences. Assume
that there exists a constant $M>0$ such that $\beta(z_k , z^\ast_k)\leq M$
for all $k\ge 1$. If $\sigma_k$ is a path measure from $z_k$ to $z^\ast_k$ with bounded hyperbolic diameter. Consider the measures $S_N = \sum_{k=1}^N \sigma_k$ and
$\sigma= \sum_{k=1}^\infty \sigma_k$. Is is clear that the Carleson
norm of both measures is bounded by a constant depending only on $M$
and $\| \sum_{k=1}^\infty (1-|z_k|) \delta_{z_k} \|_c$. Thus, the
previous lemma and remark say that $\ca(S_N) \rr \ca(\sigma)$ in
$L^2$-norm. The next result tells us
that on the unit circle the function $2 \Im \ca(\sigma)$ is an argument of the quotient $b/b^\ast$. \\

\begin{lemma}\laba{intwin}
Let $b$ (respectively $b^\ast$) be a normalized \CNBP$\ $with zero sequence $\{ z_k \}$
(respectively $\{ z^\ast_k \}$) such that\/  $\sup_k \beta(z_k , z^\ast_k) < \infty$. Then
$$
\exp\,[\,i 2\, \mbox{\em Im}\, \ca(\sigma)(\eiti)] = e^{i\gamma}b(\eiti) \ov{b^\ast (\eiti)}
$$
at almost every point $\eiti\in \partial \disc$,
where  $e^{i\gamma}= \prod_{k\ge 1} \frac{z_k}{z_k^\ast}
\frac{|z_k^\ast|}{|z_k|}$, and we are interpreting here that $z/|z|  = |z|/z = -1$ when $z=0$.
\end{lemma}
\begin{proof}
We can assume that $z_k \neq 0 \neq z^\ast_k $ for all $k\ge 1$.
Let $\alpha_z(w) = (\ov{z}/ |z|) (z-w)/(1- \ov{z} w)$. A straightforward
calculation shows that
\begin{equation}\laba{alphy}
\frac{\alpha_{z_k}(\eiti)}{\alpha_{z^\ast_k}(\eiti)}  =
\frac{z_k^\ast}{z_k} \left| \frac{z_k}{z_k^\ast} \right| \left(
\frac{1-z_k\meiti}{1-z^\ast_k\meiti} \right)^2 \left|
\frac{1-z^\ast_k\meiti}{1-z_k\meiti} \right|^2 .
\end{equation}
Observe that
$$
\int_{ z_k }^{ z^\ast_k }  \frac{2}{\eiti-z}\, dz = 2\,
[\Log(1-z_k\meiti)-\Log(1-z^\ast_k\meiti)] ,
$$
where the imaginary part of $\Log(1-w)$ varies between $-\pi/2$
and $\pi/2$ for $w\in \disc$. Hence,
$$
2 \Im \ca(S_N)(\eiti) = \Im \sum_{k=1}^N \int_{ z_k }^{ z^\ast_k }  \frac{2}{\eiti-z}\, dz
= 2  \sum_{k=1}^N [\Arg(1-z_k\meiti)-\Arg(1-z^\ast_k\meiti)] ,
$$
where $ \Arg (1-w) $ denotes the principal branch of the argument,
which for $w \in \disc$ takes values between $-\pi / 2$  and $ \pi / 2$ . According to \eqref{alphy} one deduces
\begin{equation}\laba{sigmy}
\exp\, [ i2\, \Im \ca(S_N)(\eiti) ] =
\left( \prod_{k=1}^N \frac{z_k}{z_k^\ast} \left| \frac{z_k^\ast}{z_k} \right|  \right)
b_N(\eiti) \ov{b_N^\ast (\eiti)} .
\end{equation}
Here $b_N$ (respectively $b_N^\ast$) denotes the Blaschke product
formed with the first $N$ zeros of $b$ (respectively $b^\ast$).
Since $\|\ca(S_N) -\ca(\sigma)\|_2 \rr 0$, there is a subsequence
$N_j$ such that the convergence holds pointwise almost everywhere on
$\partial \disc$, and since the same holds for the partial Blaschke
products $b_N$ and $b^\ast_N$, we can assume that the subsequence
$N_j$ achieves the three convergences at once. Furthermore, since
$\sum |z^\ast_k - z_k| < \infty$, we have that
$\prod_{k=1}^N \frac{z_k}{z_k^\ast} \left| \frac{z_k^\ast}{z_k} \right|$
converges to a certain point $ e^{i\gamma}$ of the unit circle.
Therefore, the Corollary follows by taking limits in both members of
\eqref{sigmy}.
\end{proof}

Given a Carleson-Newman Blaschke product we denote by $\mu_b$ the
Carleson measure given by $\mu_b = \sum (1-|z|) \delta_z$ where
the sum is taken over all zeros $z \in \disc$ of $b$ counting multiplicities.
Given a function $v\in L^1(\partial \disc)$, let $\tilde{v}$ be its harmonic conjugate normalized so that
$\tilde{v}(0) =0$.

\begin{coro}\laba{intw}
Let $b$ be a \CNBP$\ $with zeros $\{ z_k \}$
and $\varepsilon
>0$. Then there is $\alpha =\alpha(\varepsilon, \|\mu_b\|_c ) >0$
such that for any sequence $\{ z^\ast_k \}$ with $\sup_k \beta(z_k ,
z^\ast_k)\leq \alpha$ and its corresponding Blaschke product
$b^\ast$, there is $h\in (H^\infty)^{-1}$ such that
$$
\| b-b^\ast h \|_\infty <\varepsilon \ \ \mbox{ and }\ \ 2\Im\/
\ca (\sigma) - \widetilde{\log |h|} \in L^\infty (\partial \disc).
$$
\end{coro}
\begin{proof}
We can assume that $b$ and $b^\ast$ are normalized.
Write $\sigma = \sum \sigma_k$, where $\sigma_k$ is the path measure
in the segment from $ z_k $ to $ z^\ast_k $. Then $\|\sigma\|_c \leq
C(\alpha ) \| \mu_b\|_c$, with $C(\alpha) \rr 0$ as $\alpha \rr 0$.
Therefore, given $\eta
>0$, we can choose $\alpha=\alpha(\eta)$ small enough so that
$\|\sigma\|_c <\eta$. By Lemma \ref{cauchy-adj} one has that
$\| C (\sigma) \|_{BMO} \leq C\eta$, and the Fefferman-Stein decomposition
of $BMO$ gives $2\Im C (\sigma)= u +\tilde{v}$,
where $\|u\|_\infty +\|v\|_\infty \leq C' \eta$. Here $C'>0$ is a fixed constant. Pick $h
= e^{-i \gamma + v+i\tilde{v}}$, where $\gamma$ is the constant
appearing in Lemma \ref{intwin}. By Lemma \ref{intwin}, at
almost every point of the unit circle one has
$$
b \ov{b}^\ast h^{-1} =  e^{i2\,{\scriptsize \Im} \ca (\sigma)}
e^{-v-i\tilde{v}} =   e^{-v+iu} ,
$$
and consequently
$$
\| b-b^\ast h \|_\infty   \leq  \|h\|_\infty \| b\ov{b}^\ast
h^{-1}-1 \|_\infty \leq e^{\|v\|_\infty}  \|  e^{-v+iu}-1 \|_\infty
\rr 0
$$
as $\eta \rr 0$.
\end{proof}

\noi It is worth mentioning that two interpolating Blaschke products
could be uniformly close but still have zero sets which  are
hyperbolically far away one from the other. For instance consider
the singular inner function $s(z) = \exp((z+1)/ (z-1))$, $z \in
\disc$. For $\alpha \in \disc$ let $s_\alpha$ be its Mobius shift
given by $s_\alpha = (\alpha - s)/ (1 - \ov{\alpha} s ) $. It is
easy to see that for any $\alpha \in \disc \setminus \{0 \}$, the
function $s_\alpha$ is an interpolating Blaschke product.
It is clear that $s_\alpha$ and $s_{\beta}$ are uniformly close if $|\alpha|+ |\beta|$ is small,
but in this case, the hyperbolic distance between its zero sets is bounded below by
$\log ( \log |\alpha| /\log |\beta|)$ when $|\alpha| < |\beta|$.
Taking $\alpha= \beta^2$ we see that the distance between their respective zeros is bounded below by $\log 2$.

\begin{lemma}\laba{moddy}
For $0 \le t \le 1$, let $u_t$ be an inner function and $h_t\in (\papa)^{-1}$. Assume that  $|u_t h_t|$ varies
continuously in $L^\infty (\disc)$. Then both $|u_t|$ and $|h_t|$ vary continuously in $L^\infty (\disc)$.
\end{lemma}
\bdem
Since $|h_t(\eiti)| =|h_t(\eiti)u_t(\eiti)|$ varies continuously in
$L^\infty(\partial \disc)$ and
$$
|h_t(z)| = \exp \left( \int_0^{2\pi}
\frac{1-|z|^2}{|1-\ov{z}\eiti|^2} \log|h_t(\eiti)| \,
\frac{d\theta}{2\pi}  \right) ,
$$
then $|h_t(z)|$ varies continuously on $L^\infty(\disc)$. Thus,
$|u_t(z)| = |h_t(z)|^{-1} |h_t(z)u_t(z)|$ varies continuously on
$L^\infty(\disc)$.
\edem

We remark that the above lemma is false without taking modulus, that is, the continuity of $t \mapsto u_t h_t$ in
$\papa$ does not imply the continuity
of $t \mapsto u_t$. An example will be given in Section 4 as a consequence of Proposition \ref{cami}.

\begin{lemma}\laba{rehor}
For $0 \le t \le 1$, let $b_t$ be a \CNBP$\ $and $h_t\in (\papa)^{-1}$. Assume that  $|b_t h_t|$ varies
continuously in $L^\infty (\disc)$. Then there is a reordering $\{
z^1_k \}$ of the zeros of\/ $b_1$ such that $\sup_k \beta
(z^0_k,z^1_k) < \infty$, where $\{ z^0_k \}$ are the zeros of\/
$b_0$.
\end{lemma}
\bdem
Let $\{ z_k (t) : k =1, 2, \ldots \}$ be the
zeros of $b_t$.  By compactness it is enough to prove that given $0<\varepsilon< 1$,
for any
$t_0\in [0,1]$ there is an open neighborhood $V$ of $t_0$ in $[0,1]$
depending on $t_0$, such that whenever $t',t''\in V$, there is a
reordering of the zeros of $b_{t''}$ such that $\beta(z_k(t'), z_k(t'')) \leq \varepsilon$, for any $k\ge 1$.

Fix $t_0\in [0,1]$, since $b_{t_0}$ is a \CNBP, its zero sequence $Z(b_{t_0})$ can be split into $n(t_0)$
sequences $ Z(b_{t_0}) = S_1 \cup \ldots \cup S_{n(t_0)} $ such that
$\beta(z,w) \geq 1$ for $z, w\in S_k$ with $z\neq w$. Consider the
family of open hyperbolic disks
$\Delta_j = \{ z \in \disc : \beta (z, z_j) < \frac{\varepsilon}{4n(t_0)} \}$ for $z_j\in Z(b_{t_0})$.
We claim that any connected component of $\bigcup \Delta_j$ contains
no more than $n(t_0)$ points of $Z(b_{t_0})$.
In fact, suppose that $\O=\Delta_{j_1} \cup \ldots \cup \Delta_{j_m}$ is a maximal connected
set such that $z_{j_1}, \ldots , z_{j_m}$ belong to different sequences $S_k$
(not necessarily a component of $\bigcup \Delta_j$).
Then
\bequ\laba{diaa}
\diam_\beta \O \leq m \frac{2\varepsilon}{4n(t_0)} \leq
n(t_0) \frac{2\varepsilon}{4n(t_0)}= \frac{\varepsilon}{2} \le \frac{1}{2}.
\eequ
Now, if $z_i \in Z(b_{t_0})$ belongs to the same sequence
$S_k$ as some of the points $z_{j_l}$, say $z_{j_1}$, then for
every $z\in \O$,
$$
\beta(z_i,z) \geq \beta(z_i, z_{j_1}) - \beta(z_{j_1},z) \geq
1-\frac{1}{2}.
$$
So, $\beta(z_i,\O) \geq 1/2$ and consequently
$\beta(\Delta_i,\O) \geq \frac{1}{2} -\frac{\varepsilon}{4n(t_0)}\geq \frac{1}{4}$. Thus,
$\Delta_i$ cannot meet $\O$, which implies that $\O$ is indeed one
of the connected components of $\bigcup \Delta_j$, and that the
hyperbolic distance between two of these components is $\geq 1/4$.

Since $b_{t_0}$ is a \CNBP, there is some $\eta >0$ such that
$|b_{t_0}| \ge \eta$ in $\disc \setminus \bigcup \Delta_j$.
Let $V_{t_0} \subset [0,1]$ be a relatively open neighborhood of
$t_0$ such that $||b_t(z)|-|b_{t_0}(z)|| < \frac{\eta}{2}$
for all $z\in\disc$ and $t\in V_{t_0}$. Then
$$
\{ |b_t| < \eta/2  \} \subset   \{ |b_{t_0}| <\eta \}  \subset \bigcup \Delta_j \ \mbox{ for all $\ t\in V_{t_0}$}.
$$
Together with \eqref{diaa}, this implies that every (simply) connected component $\Om$ of the set
$\{ |b_{t_0}| <\eta \}$  has hyperbolic diameter bounded by $\varepsilon/2$.

The lemma will follow if we show that $b_{t_0}$ and $b_t$ have the
same number of zeros in $\Om$ for $t\in V_{t_0}$. By a conformal
mapping between $\Om$ and $\disc$ and the first paragraph of the
proof, it is enough to show that if $B_t$ are finite
Blaschke products such that $|B_t|$ varies continuously for $0\leq
t\leq 1$, then they have the same degree. By compactness, the
degrees are bounded and there is some $0< r <1$ such that
$Z(B_t)\subset r\disc$ for all $t$. Furthermore, composing at the
right side by some automorphism of $\disc$ we can also assume that
$B_t(0)\neq 0$ for all $t$. If we write $\alpha_j$, with $1\leq
j\leq n(t)$, for the zeros of $B_t$ counting multiplicities,
Jensen's formula gives
$$
r^{n(t)} = |B_t(0)| \prod_{j=1}^{n(t)} \frac{r}{|\alpha_j|} = \exp
\left\{ \frac{1}{2\pi} \int_{-\pi}^{\pi} \log|B_t(re^{i\theta})|
\, d\theta  \right\} .
$$
Since the right member of the equality is a continuous function of
$t$, so is $r^{n(t)}$, which means that $n(t)$ is constant.
\edem

\noi We are ready now to prove our characterization of the components of $\cni$.

\begin{theo}\laba{cra}
Let $b, b^\ast$ be \CNBP\ and $h\in (\papa)^{-1}$. Then $b$ and $ b^\ast h$ can be joined by a
path contained in $\cni$ if and only if the following two
conditions hold
\begin{enumerate}
\item[{\em (1)}]
there is a reordering $\{ z^\ast_k \}$ of the zeros of $b^\ast$ such
that $\sup_k \beta (z_k,z^\ast_k) < \infty$, where $\{ z_k \}$ are
the zeros of $b$.
\item[{\em (2)}]
if $\sigma = \sum \sigma_k$, where $\sigma_k$ is the path measure on
the segment from $z_k$ to $z^\ast_k$, then {\em
$$
2\, \Im \ca (\sigma) -  \widetilde{\log |h|} \in L^\infty(\partial
\disc) .
$$
}
\end{enumerate}
\end{theo}
\begin{proof}

Let us first discuss the sufficiency of the conditions (1) and (2).
Suppose that $\{z_k\}$ and $\{ z^\ast_k\}$ satisfy (1). Write
$M=\sup_k \beta (z_k,z^\ast_k)$ and $z_k(t) = z_k + t(z^\ast_k- z_k)$ for $0\leq t\leq 1$. If $b_t$ is the Blaschke
product with zeros $\{z_k (t) : k=1,2, \ldots \}$ then $\| \mu_{b_t}
\|_c \leq C_1 = C_1(\|\mu_b\|_c , M)$ for all $t$. Hence, there is a
constant $\varepsilon >0$ independent of $t$ such that
\bequ\laba{gyu1}
|b_t(z)| \geq \varepsilon \ \mbox{ if }\    \beta(z, Z(b_t)) \geq 1 .
\eequ
Let $\alpha = \alpha(\varepsilon /2 , C_1) <1$ be the quantifier of Corollary \ref{intw} and choose points
$0= t_0 < \cdots < t_n =1$ in the interval $[0,1]$ such that
$\beta(z_k(t), z_k(t')) < \alpha$ for all $k\ge 1$, whenever $t$ and $t'$
belong to the same interval $[t_{j},t_{j+1}]$, $j=0,1, \ldots, n-1 $.
By Corollary \ref{intw}, for each $0\leq j<n$ there is $g_{j+1}\in (\papa)^{-1}$ such that
\bequ\laba{gyu2}
\| b_{t_{j}} - b_{t_{j+1}} g_{j+1} \| < \varepsilon/2 \ \ \mbox{ and
}\ \ 2\Im \ca (\sigma_{t_{j}, t_{j+1}}) - \widetilde{ \log|g_{j+1}| }
\in L^\infty(\partial \disc) .
\eequ
Here $\sigma_{t_{j}, t_{j+1}}$ is the sum of the path measures from $z_k(t_j)$ to $z_k(t_{j+1})$.
Since $|b_{t_{j}}(z)| \geq \varepsilon$ when $\beta(z, Z(b_{t_{j}}))\geq 1$, for any $0 \leq s \leq 1$, the zeros of the
function $b_{t_{j}} + s (b_{t_{j+1}} g_{j+1} - b_{t_{j}})$ are
contained in $\Omega_j = \{z \in \disc : \beta(z, Z(b_{t_{j}})) \leq 1 \}$.
Moreover, by Rouche's Theorem on each connected component of
$\Omega_j$ it has as many zeros as $b_{t_{j}}$. Hence,
$$\{ b_{t_{j}} + s (b_{t_{j+1}} g_{j+1} - b_{t_{j}}) : 0 \leq s \leq 1 \}$$
is a segment
contained in $\cni$ which joins $b_{t_{j}}$ and $b_{t_{j+1}} g_{j+1}$.
Thus, $b$ and $b_{t_n} g_1 \ldots g_n$ can be joined by a polygonal contained in $\cni$.
Write $g= \prod_{1}^{n} g_j \in (\papa)^{-1}$ and observe that
$$
2\Im \ca (\sigma_{t_0, t_n}) - \widetilde{ \log|g| } =
\sum_{j=0}^{n-1} ( 2\Im \ca (\sigma_{t_{j}, t_{j+1}}) - \widetilde{
\log|g_{j+1}| } )    \in L^\infty(\partial \disc) .
$$
and that $\ca (\sigma) = \ca (\sigma_{t_0, t_n})$ on the unit circle . So far we have proved that if $b$ and $b^\ast$ satisfy (1)
there is $g\in (\papa)^{-1}$ that satisfies (2) and such that $b$
and $ b^\ast g$ can be joined by a path contained in $\cni$. If
$h\in (\papa)^{-1}$ is any function that satisfies (2) then
$$
\widetilde{ \log \frac{|h|}{|g|} } = \widetilde{ \log |h| } -
\widetilde{ \log |g| } \in L^\infty( \partial \disc) ,
$$
which implies that $h= g e^f$ for $f\in \papa$. This means that $h$
and $g$ are in the same connected component in $(\papa)^{-1}$.
This proves the sufficiency.

The necessity of (1) follows from Lemma \ref{rehor}. Let us now
prove that (2) is also necessary. Let $\gamma: [0,1] \rightarrow  \cni$ be a path joining $\gamma(0) = b$ and $\gamma (1) =
b^\ast h$. Thus $\gamma (t) = b_t h_t$, where $b_t$ is a \CNBP$\ $and $h_t \in (\papa)^{-1}$. By
compactness, there exists a constant $K\ge 1$ such that
$K^{-1} \le |h_t (z)| \le K$ for any $z \in \disc$ and any $0 \leq t \leq 1$. Let $0< \varepsilon < K^{-2}$
be a number satisfying \eqref{gyu1}
and choose points $0=t_0 < t_1 < \ldots < t_n =1$ that simultaneously satisfy \eqref{gyu2}  and
$$
\|b_t h_t - b_s h_s \|_\infty < \frac{\varepsilon}{2K}
$$
for any $t, s \in [t_{j-1} , t_{j}]$ and  $j=1,2, \ldots, n$.
The existence of such points follows from Corollary \ref{intw} and the first paragraph in the proof of Lemma \ref{rehor}.
Hence, on $\partial \disc$ we have
$$
\| b_{t_{j}} \ov{b}_{t_{j+1}} -  g_{j+1} \|_\infty < \varepsilon/2 \ \ \mbox{ and }\ \
\|b_{t_{j}} \ov{b}_{t_{j+1}} - h^{-1}_{t_{j}} h_{t_{j+1}} \|_\infty < \varepsilon/2 ,
$$
which leads to
$$
\|g_{j+1} h_{t_{j}} h^{-1}_{t_{j+1}} -1 \|_\infty \le K^2 \|g_{j+1} - h^{-1}_{t_{j}} h_{t_{j+1}} \|_\infty < K^2 \varepsilon <1.
$$
Consequently, $g_{j+1} h_{t_{j}} h^{-1}_{t_{j+1}} = e^f$ with $f\in \papa$, which means that
$$
\widetilde{ \log|g_{j+1}| } - (\widetilde{ \log|h_{t_{j+1}}| }      -           \widetilde{ \log|h_{t_{j}}| } )
 \in L^\infty(\partial \disc) .
$$
Summing from $j=0$ to $n-1$ we get
$$
\widetilde{ \log|g| } - \widetilde{ \log|h| } =
\sum_{j=0}^{n-1}\widetilde{ \log|g_{j+1}| } - (\widetilde{ \log|h_{t_{n}}| }   -    \widetilde{ \log|h_{t_{0}}| } )
 \in L^\infty(\partial \disc) .
$$
Since $g$ satisfies (2), so does $h$.
\end{proof}

\noi
It is important to notice that the invertible function $h$ of the above theorem is associated to the particular
reordering of the zeros of $b^\ast$. Indeed, there could exist two different reorderings of $\{ z_k^\ast\}$
satisfying condition (1) of the theorem that lead to respective functions $h_1 , h_2 \in (\papa)^{-1}$ with
$$
\widetilde{\log |h_1|}-\widetilde{\log |h_2|} \not \in L^\infty (\partial \disc).
$$
An example of this phenomenon is given in Section 4, where in addition $b=b^\ast$.
We also remark that instead of taking segments in the above proof, we can use an equicontinuous
family (with respect to the hyperbolic metric) of curves joining $z_k$ with $z^\ast_k$ with bounded hyperbolic length.
In the proof of the theorem we have also showed the following

\begin{coro}\laba{ultim}
Let $b$ and $b^\ast$ be two \CNBP. Then
there exists a function $h \in (\papa)^{-1}$ such that $b$ and
$b^\ast h$ can be joined by a continuous path contained in $\cni$
if and only if there is  a reordering $\{ z^\ast_k \}$ of the zeros
of $b^\ast$ such that $\,\sup_k \beta (z_k,z^\ast_k) < \infty $,
where $\{ z_k \}$ are the zeros of $b$.
\end{coro}

\noi
We end this section applying Theorem \ref{cra} to the little Bloch space.
The following lemma is well known. It follows immediately from a couple of results given by Guillory, Izuchi and Sarason:
Theorem 1 of \cite{g-i-s} and the first theorem in Section 3 of \cite{guisa}.

\begin{lemma}\laba{litbb0}
Let $u$ be an inner function and let  $b$ be a \CNBP\
with zeros $Z(b)$. The following two conditions are equivalent
\begin{enumerate}
\item[{\em (1)}] $\lim_{|w| \to 1} \sup \{|u(w)| : \beta(w, Z(b)) < \alpha \} =0$ for any $\alpha>0$,
\item[{\em (2)}] $\lim_{|w| \to 1} |u(w)| (1- |b(w)|) =0$.
\end{enumerate}
In particular, if these conditions hold, \/  $\sup \{ | |u(z)|-|b(z)| | : z \in \disc \} =1$.
\end{lemma}

\begin{theo}\laba{litbloch}
Let $u$ be an inner function in the little Bloch space that is not a
finite Blaschke product and let $\Om$ be a connected component of
$\cni$. Then $|u|$ cannot be
approximated uniformly in $\disc$ by functions $|f|$, with
$f\in\Om$. In particular, $u$ does not belong to the closure of any
component of\/ $\cni$.
\end{theo}
\bdem We argue by contradiction.  Assume there exists a sequence $b_n$ of
\CNBP $\ $and a sequence of functions $h_n \in (\papa)^{-1}$ such that $b_n h_n \in \Om$ and
$
\sup \{ | |u(z)| - |b_n (z) h_n (z)| | : z \in \disc \}  \to 0
$.
Then $|h_n|$ tend to $1$ uniformly on $\partial \disc$, and since
$h_n$ are invertible, it follows that $|h_n|\rr 1$ uniformly on $\disc$. Hence,
$$
\sup \{ | |u(z)| - |b_n (z) | | : z \in \disc \}  \to 0.
$$
Therefore, there is $n_0$ such that $| |u (z)|-|b_{n_0} (z)| | < 1/2$ for all $z \in \disc$. Consequently, Lemma \ref{litbb0} says
that there are constants $m>0$, $\eta >0$ and a subsequence $\{ z_k\}$ of zeros of $b_{n_0}$ such that
$\,\sup \{  |u(z)|  : \beta(z,z_k) \le m \} > \eta$ for all $k\ge 1$.
Since $u$ is in the little Bloch space, $|u(z_k)| \ge \eta /2$ for all $k$ sufficiently large.
Now fix $n$ such that $| |u(z)|- |b_n (z)|  | < \eta/4$ for all $z \in \disc$.
In particular, $|u(w)| <  \eta/4$ for any zero $w$ of $b_n$.
Since $u$ is in the little Bloch space,
$$ \beta(z_k, Z(b_n)) \ge
\beta\big( \{z: |u(z)| \ge \eta/2 \mbox{ and } |z| \ge |z_k| \} ,   \{z: |u(z)| \le \eta/4\} \big) \to \infty
$$
when $k\rr \infty$.
By Theorem \ref{cra}
there is no $h\in \papai$ such that $b_n h$ and $b_{n_0}$ connect in
$\cni$, which is a contradiction. \edem

\section{On the components of $\inni$}

\noi
We say that a Blaschke product $b$ is floating if there is
a sequence $0<r_n < 1$, tending to $1$,  such that
$\inf_\theta |b(r_n \eiti)| \rr 1$ as $n \to \infty$.

\begin{lemma}\laba{floatdecom}
Every Blaschke product $b$ can be factorized as $b=b_1b_2$, where
$b_1$ and\/ $b_2$ are floating Blaschke products.
\end{lemma}
\begin{proof}
Let $Z(b)$ be the zeros of $b$ counting multiplicities. Given any $0<r<1$ and $\beta <1$,
there are constants $r_0 , r_1$, with $r< r_0 <r_1 <1$, such that if
$B_0$ is the Blaschke product whose zeros are the zeros of $b$ that
lie in $\{ |z| \le r \}\cup \{ |z| \ge r_1 \}$, then $\inf_\theta
|B_0(r_0 e^{i\theta})| > \beta$.
Thus, if $0< \beta_k <1$ is a sequence tending to $1$ and
$0<r_1<1$ is given, we can inductively construct a sequence
$r_k<r_{k+1} \to 1$, such that if $B_k$ is the Blaschke product whose zeros are those of $b$ that lie in
$\{ |z| \le r_{k-1} \}\cup \{ |z| \ge r_{k+1} \}$, then
$\inf_\theta |B_k(r_k e^{i\theta})| > \beta_k$ for all $k>1$. Define $b_1$ and $b_2$ as the
Blaschke products whose zeros are respectively
\begin{align*}
Z(b_1)&= \{ z\in Z(b):\ \mbox{$|z|\le r_1\,$ or $\,r_{4k+3} \le |z| \le r_{4k+5}, \,$ for $k\ge 0$}\},\\
Z(b_2)&= \{z\in Z(b):\ \mbox{$r_{4k+1} < |z| < r_{4k+3}, \,$ for
$k\ge 0$}\}.
\end{align*}
Then $|b_1 (z)| > \beta_{4k+2}$ if $|z|=r_{4k+2}$ and $|b_2 (z)| > \beta_{4k}$ if $|z|=r_{4k}$
for all $k\ge 1$. It is also clear that $b=b_1b_2$.
\end{proof}

\noi We say that an open set $G\subset \disc$ is non-tangentially dense if for almost every $e^{i\theta}\in\partial\disc$,
$G$ contains truncated cones
$$
\Lambda_\alpha^r (e^{i\theta}) = \{ z\in\disc : \, |z-e^{i\theta}| < \alpha (1-|z|), \ |z|>r \} , \ \
r<1<\alpha
$$
of arbitrarily large opening $\alpha$. Since an inner function $u$
has non-tangential limits of modulus $1$ at almost every point of $\partial\disc$, the set
$\{ z\in\disc : \, |u(z)|>\delta \}$ is non-tangentially dense for any $0<\delta<1$.
Next we state several technical results that will be used in the proof of our main theorem.

\begin{propos}\laba{modone2}
Let $u_0$ be a floating Blaschke product and $u_1$  be an inner function. Assume that there exist a function
$h \in (\papa)^{-1}$ with $\|h \|_\infty \leq 1$, an open set\/
$\Om\subset\disc$ and a constant $0 < \delta < 1$ such that
\begin{enumerate}
\item[{\em (1)}] For \/ $i=0,1$, one has $\Om \subset \{ |u_i|< \delta \}$
\item[{\em (2)}] Arclength $ \lambda_{\partial \Om} $ on $\partial \Om$  is a Carleson measure.
\item[{\em (3)}] There exists an analytic branch  of the logarithm of\/
$ u_1 h /u_0 $ in an open set of the unit disk containing\/ $\disc \setminus \Om$, which we denote by\/ $\log(u_1 h / u_0)$,
whose non-tangential limits
$$
\lim_{z \in \disc \setminus \Om ,\, z \to e^{i \theta} } \log(u_1 h /
u_0) (z)
$$
 exist at almost every point $e^{i \theta} \in \partial \disc$  and define a function in
$L^\infty (\partial \disc)$
\item[{\em (4)}] $10 \,\delta  \| \lambda_{\partial \Om} \|_c \le \inf_{\partial \disc}
|h| $.
\end{enumerate}
Then\/ $u_0 $ and $u_1 h$ can be joined by a path contained in $\inni$.
\end{propos}
\bdem First observe that for any $t\in [0,1]$, the function $g_t = u_0 \exp( t\log (u_1 h / u_0) ) $ is a bounded analytic function on a
neighborhood of $\disc \setminus \Om$. Moreover,
$$\left|   u_0 e^{ t\log\left( \frac{u_1 h}{u_0} \right) }  \right|
=  |u_0|^{1-t} \, |u_1|^t \,  |h|^t .
$$
Observe that $|h| \leq |g_t| \leq 1$ on the unit circle and $ |g_t| \leq \delta$ on $\partial \Om$. Fix $0\le t\le 1$. By
duality (see \cite[IV, Thm.$\,$1.3]{gar}), one has
\begin{align*}
\dist_{L^\infty(\partial \disc)} ( g_t , \papa )
&= \sup_{F\in H_0^1, \, \|F\|_1 \leq 1} \left| \int_0^{2\pi}
g_t (e^{i \theta}  )  F (e^{i \theta})  \frac{d\theta}{2\pi} \right|
\\*[1mm]
 &= \sup_{F\in H^1, \,
\|F\|_1 \leq 1} \left| \int_{\partial \disc}
g_t(z)  F(z) \frac{dz}{2\pi }
\right|
\end{align*}
Fix $F \in H^1$. Cauchy's Theorem and a limit argument shows that
\bequ\laba{taq}
\int_{\partial \disc}
g_t (z)  F (z) \frac{dz}{2\pi } =  \int_{\partial \Om} g_t (z) F (z)
\frac{dz}{2\pi }
\eequ
Indeed, since $u_0$ is a floating Blaschke product, there are $r_j\rr 1$ such that $\inf_{|z|=r_j}|u_0(z)|  \rr 1$.
By condition (1) the circles $|z|= r_j$ do not meet $\Omega$ if $j$ is sufficiently large. Let $\Om_k$, $k\ge 1$, be the connected components of $\Om$.
By Cauchy Theorem,
$$
 \int_{\partial (r_j\disc) } g F \frac{dz}{2\pi }
 =
\sum_{\Om_k \subset r_j \disc} \int_{\partial \Om_k}  g  F
\frac{dz}{2\pi }.
$$
By condition (2),
$$
\lim_{j\rr \infty} \int_{\partial (r_j\disc) } g F \frac{dz}{2\pi
} = \int_{\partial \Om}  g  F \frac{dz}{2\pi }.
$$
Now,
$$
\int_{\partial (r_j\disc) } g(z) F(z) \frac{dz}{2\pi } =
\int_{\partial \disc } g(r_jw) F(r_jw)r_j \, \frac{dw}{2\pi } \,
\rr \, \int_{\partial \disc } g(w) F(w) \, \frac{dw}{2\pi }
$$
by the dominated convergence theorem, observing that $|g(r_jw)F(r_jw) |$ is bounded by the non-tangential maximal function of $F$ at $w$. This proves \eqref{taq}.
Hence,
$$
\left| \int_{\partial \disc} g_t (z)  F (z) \frac{dz}{2\pi }  \right| \le  \delta  \| \lambda_{\partial \Om} \|_c \|F \|_1.
$$
Consequently, (4) says to $\dist_{L^\infty(\partial \disc)} ( g_t  , \papa ) < \inf_{\partial \disc} |h| / 10$.
Therefore, there is $f_t \in \papa$ such that
 \bequ\laba{tlog} \|g_t - f_t\|_{L^\infty(\partial \disc)} \leq \frac{1}{5} \inf_{\partial \disc}
|h| , \eequ implying that at almost every point of $\partial\disc$
one has

\bequ\laba{molog}
 |f_t| \ge |h|^t - \frac{|h|}{5} \ge
|h|-\frac{|h|}{5}  = \frac{4}{5} |h| \eequ
In particular, $f_t \in \inni$ for every $t\in[0,1]$.

Since $\log( u_1 h / u_0) \in L^\infty(\partial \disc)$, the mapping
from $[0,1]$ to $L^\infty (\partial \disc)$ given by $t \mapsto e^{
t\log ( u_1 h / u_0 ) }$ is continuous, and consequently there is a
finite partition of $[0,1]$, $0=t_0 < t_1 < \cdots < t_n =1$,
such that
$$
\| u_0 e^{ t_j \log\left( \frac{u_1 h}{u_0} \right) } - u_0 e^{ t_{j+1} \log\left( \frac{u_1 h}{u_0} \right) }
\|_{L^\infty(\partial \disc)} < \frac{1}{5}\inf_{\partial \disc}
|h|,
$$
which together with \eqref{tlog} implies that the three quantities
$$
\| u_0  - f_{t_0} \|_{L^\infty(\partial \disc)}, \ \ \| f_{t_j} -
f_{t_{j+1}} \|_{L^\infty(\partial \disc)}, \ \ \| u_1 h  - f_{t_{n}}
\|_{L^\infty(\partial \disc)}
$$
are bounded above by $ \frac{3}{5} \inf_{\partial \disc} |h|$ for $0\leq
j<n$. So, for any function in the segment joining $f_{t_j}$ with
$f_{t_{j+1}}$, that is for any $0 \leq s \leq 1$, \eqref{molog} says
that
$$
| f_{t_j}+ s ( f_{t_{j+1}} - f_{t_j} ) | \geq  | f_{t_j}| -  |
f_{t_{j+1}} - f_{t_j}| \geq  \frac{4}{5} \inf_{\partial \disc} |h| -
\frac{3}{5} \inf_{\partial \disc} |h| = \frac{1}{5} \inf_{\partial
\disc}|h| \, ,
$$
and the same holds for the segments joining $u_0$ with $f_{t_0}$,
and $f_{t_n}$ with $u_1 h$. Hence, all these segments are contained
in $\inni$ and their union is a path in $\inni$ between $u_0$ and $u_1 h$.
\edem

We shall also use the following version of Proposition \ref{modone2}.

\begin{propos}\laba{modone3}
Let $u_0$ be a floating Blaschke product and $u_1$  be an inner function.  Suppose that  \/ $\Om \subset \disc$ is
an open set and\/ $0< \delta < 1$ is a constant satisfying properties {\em (1)} and {\em (2)} of Proposition \ref{modone2}.
Instead of\/ {\em (3)} and {\em (4)} assume that
\begin{enumerate}
\item[{\em (3')}]  There exists an analytic branch  of the logarithm of the function $u_1 /u_0$
in an open set of the unit disk containing $\disc \setminus \Om$, which we denote by $\log( u_1 / u_0 )$,
whose non-tangential limits
$$
\lim_{z \in \disc \setminus \Om , z \to e^{i \theta} } \log(u_1  /
u_0) (z)
$$
exist at almost every point $e^{i \theta} \in \partial \disc$ and
define a function in $\bmo(\partial \disc)$.
\item[{\em (4')}] $10 \,\delta  \| \lambda_{\partial \Om} \|_c \le e^{-2\|\log( u_1 / u_0 )\|_{BMO}}$.
\end{enumerate}
Then there is $h\in (\papa)^{-1}$ with $\|h \|_\infty \leq 1$ such that
{\em (3)} and {\em (4)} of Proposition \ref{modone2} hold.
\end{propos}
\bdem Write $\gamma = \| \log( \frac{u_1 }{u_0} )\|_{BMO}$. By
hypothesis, on the unit circle one can decompose
$$
\log( \frac{u_1 }{u_0} )=\Im \log( \frac{u_1 }{u_0} )= r+ \tilde{s},
\ \ \mbox{ with }\ \ \|r\|_\infty + \|s\|_\infty \le \gamma .
$$
Taking  $h= e^{-(s+\gamma +i\tilde{s})}$ we have $e^{-2\gamma}\le |h| \le
1$. Thus (4') implies (4). Define $\log(u_1 h / u_0) = \log(u_1 /
u_0) - (s+\gamma +i\tilde{s})$. Therefore, at almost every point of the unit
circle one has
$$ \log( \frac{u_1 h }{u_0} ) = -(s+\gamma) +i r
$$
and (3) of Proposition \ref{modone2} holds.  Observe also that
$\|\log( \frac{u_1 h }{u_0} )\|_{L^\infty(\partial \disc)} \leq
3\gamma$. \edem

In certain cases, at almost every point of the unit circle the
logarithm of the quotient of two Blaschke products can be written as
a Cauchy integral of a Carleson measure.

\begin{lemma}\laba{logform}
Let $u, b$ be  Blaschke products. Let $\Om$ be an open set of the unit disk containing all the zeros of $u$ and $b$ such that
$\disc \setminus \ov{\Om}$ is non-tangentially dense.
Let $\nu= \nu_u-\nu_b$, where
$\nu_u$ (respectively $\nu_b$) is the sum of the harmonic measures $\omega (z,- , \Om)$ on $\partial \Om$
from the zeros $z$ of $u$ (respectively $b$).
Suppose that the boundary $\Gamma_j$ of each connected component $\Om_j$ of $\Om$ is a Jordan
rectifiable curve with $0 \not \in \Gamma_j $ satisfying
\begin{enumerate}
\item[{\em (1)}] Both functions $u$ and $b$ have the same finite number of zeros in each $\Om_j$
\item[{\em (2)}] Arclength on the union of\/ $\Gamma_j$ is a Carleson measure
\item[{\em (3)}] There is a
constant $C_0 >0$ such that for any arc $\gamma \subset \Gamma_j$, $| \nu (\gamma)| < C_0 $.
\end{enumerate}
For each  $j\geq 1$ fix a point $\xi_j \in \Gamma_j$.
Then, there is a constant $C_1$ such that for any  $z \in \disc \setminus \ov{\Om}$,
$$
\log \frac{u}{b}(z) = C_1 - \sum_{j} \int_{\Gamma_j} \nu \big(
\gamma(\xi_j,\xi) \big)   \frac{d\xi }{\xi-z} - \sum_{j}
\int_{\Gamma_j} \nu \big( \gamma(\xi_j,\xi) \big)
\frac{d\ov{\xi}}{(1-\ov{\xi}z)\ov{\xi}}
$$
where $\gamma(\xi_j,\xi)$ denotes the arc contained in $\Gamma_j$
which goes from  $\xi_j$ to $\xi$ in the counterclockwise direction, defines a logarithm of\/ $u/b$ in\/
$\disc\setminus \ov{\Om}$.
Moreover, there exists a constant $C_2$ such that
{\em
$$
\log \frac{u}{b} (z)=
C_2 + 2i \,\Im  \ca \left[ \sum_{j} \nu \big( \gamma(\xi_j,\xi) \big)  \, d\xi|_{\Gamma_j}  \right] (z)
$$
}
for almost every $z \in \partial \disc$.
\end{lemma}
\bdem \noi Let $\varphi_\xi (z)= (\xi - z) / (1- \ov{\xi}z)$. For any $z \in \disc \setminus \Om$, the function
$$
\frac{\varphi'_\xi}{\varphi_\xi}(z)= \frac{|\xi|^2-1}{ (1-\ov{\xi}z)
(\xi-z) } = \frac{\ov{\xi}}{ (1-\ov{\xi}z)  }+ \frac{1}{ (z- \xi) }
$$
is harmonic with respect to $\xi$ in the interior of $\Gamma= \cup \Gamma_j$. Then
$$
\frac{u'(z)}{u(z)}=\sum_{u(\xi_n)=0}\frac{\varphi'_{\xi_n}(z)}{\varphi_{\xi_n}(z)}
=\int_\Gamma \frac{\varphi'_\xi (z)}{\varphi_\xi (z)} \,
d\nu_u(\xi),
$$
and the same holds for $b$. So, for any $z \in \disc \setminus
\ov{\Om} $,
\begin{align}\laba{loggy}
\frac{u'(z)}{u(z)} - \frac{b'(z)}{b(z)} =  \int_\Gamma \frac{1}{ (z-\xi) }
+\frac{\ov{\xi}}{ (1-\ov{\xi}z) } \ d\nu(\xi)
\end{align}
On the other hand, for $j=1,2, \ldots $ and $z \in \disc \setminus \ov{\Om}$,
\begin{eqnarray}\laba{primb}
\lefteqn{ \! \!
\frac{d}{dz}   \int_{\Gamma_j} \left[ \int_{\gamma(\xi_j,\xi)}
\frac{dv}{v-z} + \int_{\gamma(\xi_j,\xi)} \frac{1}{(1-\ov{v}z)}
\frac{d\ov{v}}{\ov{v}}\right] d\nu(\xi)    } \nonumber \\*[2mm]
& = &  \int_{\Gamma_j} \left[ \int_{\gamma(\xi_j,\xi)} \frac{dv}{(v-z)^2} + \int_{\gamma(\xi_j,\xi)}
\frac{d\ov{v}}{(1-\ov{v}z)^2}  \right] d\nu(\xi) \nonumber \\*[2mm]
& = &  \int_{\Gamma_j} \left[ \frac{1}{(z-\xi)} + \frac{\ov{\xi}}{(1-\ov{\xi}z)} \right] d\nu(\xi),
\end{eqnarray}
because $\int_{\Gamma_j}  d\nu=0$.
From \eqref{loggy} and  \eqref{primb}, we deduce that  there exists
a constant $C_1$ such that on $ \disc \setminus \ov{\Om}$ the function
$$
  \log \frac{u}{b}(z) = C_1 +   \sum_{j} \int_{\Gamma_j}
\left[ \int_{\gamma(\xi_j,\xi)} \frac{dv}{v-z}
+\int_{\gamma(\xi_j,\xi)}\frac{1}{(1-\ov{v}z)}
\frac{d\ov{v}}{\ov{v}}\right] d\nu(\xi)
$$
is a logarithmic branch of $u/b$.
Using Fubini, one gets
$$
\log \frac{u}{b}(z) =  C_1 - \sum_{j} \int_{\Gamma_j} \nu \big(
\gamma(\xi_j,v) \big) \frac{dv }{v-z} - \sum_{j}
\int_{\Gamma_j} \nu \big( \gamma(\xi_j,v) \big)
\frac{d\ov{v}}{(1-\ov{v}z)\ov{v}} \, ,
$$
given that $\chi_{\gamma(\xi_j,v)} (v) = \chi_{\gamma(v, \xi_j)} (v)$ and $\nu (\Gamma_j) =0$.
This gives the first statement. To prove the second identity, consider the functions
\begin{align*}
f (z) &=
\sum_{j} \int_{\Gamma_j} \nu \big( \gamma(\xi_j,v) \big)
\frac{\ov{z}dv }{1-\ov{z}v}  ,\\*[1mm]
g (z) &=  \  C_1 -\sum_{j} \int_{\Gamma_j}   \nu \big( \gamma(\xi_j,v) \big)
\frac{d\ov{v}}{(1-\ov{v}z)\ov{v}}  ,
\end{align*}
which according to Lemma \ref{cauchy-adj}, $\ov{f} \in H^2_0$ and $g\in H^2$.
Observe that $\log(u/b)= f+g$  on $\partial \disc$. Since  $\log |u/b| = 0= \Re (f+g)$,
the real part of the function $g+\ov{f}\in H^2$ vanishes. Hence, $g
= -\ov{f}+ic$, where  $c\in \mathbb{R}$ is a constant, meaning that
at almost every point of the unit circle,
$$
\log \frac{u}{b} = f -\ov{f}+ic = 2i\, \Im f + ic .
$$
\edem

\noi
Given a Blaschke product $u$, we will construct an interpolating
Blaschke product $b$ and a Carleson contour $\Gamma=\partial \Om$ verifying Lemma
\ref{logform}.
The system of rectifiable Jordan curves $\Gamma_j $ appearing in
Lemma \ref{logform} is presented in the following result which is
 part of the proof of \cite[Lemma 3.2]{n-s}. An explicit proof can
be found in \cite[Lemma 2]{h-n}.
This is a variation of the classical corona construction given by Carleson in \cite{ca.c}.

\begin{lemma}\laba{contorn}
Let $u \in \papa$ with $\|u \|_\infty=1$.  Let  $0<\delta <1$ be a fixed constant. Then there exist a constant $\varepsilon =
\varepsilon (\delta ) >0$ and a system $\Gamma = \cup \Gamma_j$
of disjoint rectifiable Jordan curves $\Gamma_j$ such that
\begin{enumerate}
\item[{\em (a)}] $|u(z)| \le \delta$ when {\em $\beta(z,\interi \Gamma) \le 1$}
\item[{\em (b)}] $\sup \{ |u(w)| : \beta(w,z) \le 15 \} >  \varepsilon$ when {\em $z\not\in \interi \Gamma$}
\item[{\em (c)}] The arclength on\/ $\Gamma$ is a  Carleson measure $\lambda_{\Gamma}$
with  $\|\lambda_{\Gamma}\|_c \le C$, where
$C$ is a universal constant independent of $u$ and $\delta$.
\end{enumerate}
\end{lemma}

\begin{lemma}\laba{trossos}
Let $u$ be a Blaschke product and\/ $\Gamma$ be a Jordan curve contained in $\disc$.
Let\/ {\em $\interi \Gamma$} denote the interior of\/ $\Gamma$, and consider the sum of harmonic measures
{\em
$$
\nu_{\Gamma} = \sum_{z\in \mbox{\scriptsize int}\, \Gamma, \, u(z)=0}\omega(z , \, - \, , \interi \Gamma) .
$$
}
If\/ $L\subset \Gamma$ then
$$
 \mbox{\em diam}_\rho L \geq   (\inf_L |u|)^{1/ \nu_{\Gamma} (L)} ,
$$
where\/ $\mbox{\em diam}_\rho L = \sup \{ \rho(z,w) : z,w \in L \}$.
\end{lemma}
\bdem By harmonicity
$$
\om (z, L, \interi \Gamma)    \,       \log (\diam_\rho L)^{-1} \leq
\log |\varphi_w(z)|^{-1}
$$
for $z\in\interi\Gamma$ and $w\in L$. Summing on $z
\in Z(u)$ we obtain
$$
\nu_\Gamma (L)    \, \log (\diam_\rho L)^{-1}  \leq \log \, (\inf_L |u|)^{-1}    .
$$
\edem
\noi We are ready now to prove the main result of the paper.

\begin{theo}\laba{components}
Let $u$ be an inner function. Then there exists a path
$\gamma: [0,1] \rightarrow \inni$ such that $\gamma (0) = u$ and
$\gamma (1) = bh$, where $b$ is a \CNBP\ and
$h\in (\papa)^{-1}$.
\end{theo}
\begin{proof}
Using a Mobius transformation we can assume that $u$ is a Blaschke
product. By Lemma  \ref{floatdecom} we can assume that
$u$ is a floating Blaschke product. Let $0< \delta < 1$ be a small
constant to be chosen later. Consider the contour
$\Gamma$ given by Lemma \ref{contorn} and decompose $u$ as $u=u_1
u_2$ into two Blaschke products $u_1 , u_2$, where $u_1$ is formed
with the zeros $z$ of $u$ that lie inside the interior of $\Gamma$ such that $\beta (z,\Gamma) >1$.
For each zero $z$ of $u_2$, part (b) of Lemma \ref{contorn} provides
a point $w \in \disc$ such that $\beta(z,w) \le 16$ and $|u_2 (w)| \ge |u(w) | > \varepsilon(\delta)$. This implies that
$u_2$ is a Carleson-Newman Blaschke product. For each component
$\Gamma_k$ of $\Gamma$ consider the measure
$$
d\nu_{u_1} (\xi) = \sum_{k \ge 1} \sum_{ u_1(z)=0} \omega (z, \xi, \text{Int}\,
\Gamma_k), \quad  \xi \in \Gamma ,
$$
where $\omega(z, \xi, \Omega)$ denotes the harmonic measure from a
point $z \in \Omega$ in the domain $\Omega \subset \disc$.
Hence, the total mass $\nu_{u_1} (\Gamma_k)$ is the number of zeros of $u_1$ in the
interior of $\Gamma_k$, which is finite by (a) of Lemma \ref{contorn}, given that $u$ is floating.
Split each $\Gamma_k$ into closed arcs that are pairwise disjoint except for the extremes
$\{\Gamma_{k,i} : 1\le i\le  \nu_{u_1} (\Gamma_k)  \}$, with $\nu_{u_1} (\Gamma_{k,i}) = 1$ for all $i$,
and locate a point $w_{k,i}$ in $\Gamma_{k,i}$. Let $b_1$ be the Blaschke product
with zeros $\{w_{k,i} : 1\le i\le  \nu_{u_1} (\Gamma_k)   ,\  k \geq 1\}$. Part (c) of Lemma \ref{contorn} and Lemma
\ref{trossos} show that $b_1$ is a \CNBP.

The theorem will follow if we show that the functions $u$ and $b_1u_2$ satisfy the four conditions of
Proposition \ref{modone3} when $\delta$ is sufficiently small.
Applying Lemma \ref{logform} to $u_1$ and $b_1$, we see that at almost every point of $\partial \disc$,
$$
\log(u_1 / b_1)= C_2 + 2i \Im \ca (\sum_j \nu(\gamma(\xi_j , \xi))
d\xi|_{\Gamma_j} )  ,
$$
where $\nu = \nu_{u_1} - \nu_{b_1}$.
By (c) of Lemma \ref{contorn} and Lemma  \ref{cauchy-adj}, $\log(u_1 / b_1)$ belongs to $\bmo
(\partial \disc)$, where $\|\log(u_1 / b_1) \|_{BMO}$ is bounded by an
absolute constant (independent of $u$ and $\delta$).
Since $ u / u_2 b_1 = u_1 / b_1$, only (1) of Proposition \ref{modone2} remains to be proved.
This will follow if we show that there is a constant $c(\delta)$ such that
\bequ\laba{apumante}
\sup_{\Gamma}  |u_2  b_1 | \leq c(\delta) \to 0 \ \mbox{ when $\delta \to 0$}.
\eequ
Fix $z \in \disc $ with $\beta(z, \text{Int} \, \Gamma ) \ge 1$ and observe that
\bequ\laba{todo}
\log \frac{1}{|u_1 (z)|} = \int_{\Gamma} \log \frac{1}{|\varphi_{w}(z)|} \, d\nu_{u_1 }(w).
\eequ
Split the integral over $\Gamma$ as integrals over $\Gamma_{k,i}$ and consider the families of short and long arcs defined by
$$
\mathcal{S} =\{\Gamma_{k,i} : \diam_\beta (\Gamma_{k,i}) \leq 1/4 \}
\, \ \mbox{  and  }\ \,
\mathcal{L} =\{\Gamma_{k,i} : \diam_\beta (\Gamma_{k,i}) > 1/4 \}.
$$
Fix $\Gamma_{k,i} \in \mathcal{S}$. Since $\beta(z, \text{Int} \, \Gamma ) \ge 1$, for $w, w_{k,i} \in \Gamma_{k,i}$,
\bequ\laba{cunte}
\log \frac{1}{|\varphi_{w} (z)|} < C_1 (1-|\varphi_{w} (z)|^2) < C_2 (1 - |\varphi_{w_{k,i}} (z)|^2 ) < 2C_2 \log \frac{1}{|\varphi_{w_{k,i}}(z)|},
\eequ
where $C_1$ and $C_2$ are universal constants.
Hence
\bequ\laba{cortos}
\int_{\Gamma_{k,i}} \log
\frac{1}{|\varphi_{w} (z)|} \, d\nu_{u_1 }(w) < 2C_2 \log \frac{1}{|\varphi_{w_{k,i}} (z)| }
\eequ
Now for each  $\Gamma_{k,i} \in \mathcal{L}$ let $\alpha_{k,i} = \alpha_{k,i} (z) \in \Gamma_{k,i} $ such that
$$
\log \frac{1}{|\varphi_{\alpha_{k,i}} (z)|} = \sup \left\{ \log \frac{1}{|\varphi_{w} (z)|} : w \in \Gamma_{k,i} \right\} \, .
$$
Clearly,
$$
\sum_{\Gamma_{k,i}\in\mathcal{L}}   \int_{\Gamma_{k,i}} \log
\frac{1}{|\varphi_{w} (z)|} \, d\nu_{u_1 }(w) \leq \sum_{\Gamma_{k,i}\in\mathcal{L}}   \log \frac{1}{|\varphi_{\alpha_{k,i}} (z)| }
\leq C_1 \int_{\disc} \frac{1-|z|^2}{|1-\ov{\xi}z|^2} dm(\xi),
$$
where $C_1$ is the universal constant in \eqref{cunte} and $m= \sum (1-|\alpha_{k,i}|^2) \delta_{\alpha_{k,i}} $.
Next we will show that $m$ is a Carleson measure whose Carleson norm is bounded independently of $z$.
Let $Q$ be a Carleson square. If $\alpha_{k,i} \in Q$, since $\Gamma_{k,i}$ is long,  then $1- |\alpha_{k,i}|^2 < C \ \text{length} (\Gamma_{k,i} \cap 2Q) $, where $C$ is a universal constant.
Thus $m(Q) < C \ \text{length} (\Gamma \cap 2Q)$.
Hence $m$ is a Carleson measure whose Carleson norm is bounded by a fixed multiple of the Carleson norm of
the arclength of $\Gamma$. Therefore, by \cite[VI, Lemma 3.3]{gar},
\bequ\laba{largos}
\sum_{\Gamma_{k,i}\in\mathcal{L}}   \int_{\Gamma_{k,i}} \log
\frac{1}{|\varphi_{w} (z)|} \, d\nu_{u_1 }(w) \leq K,
\eequ
where $K$ is another universal constant. Applying \eqref{cortos} and \eqref{largos} in \eqref{todo} we get
$$
\log \frac{1}{|u_1 (z)u_2(z)|} \le \log \frac{1}{|u_2(z)|} + 2C_2 \sum_{\Gamma_{k,i}\in\mathcal{S}} \log \frac{1}{|\varphi_{w_{k,i}} (z)| } +K ,
$$
for $\beta(z,\interi\Gamma) \ge 1$.
Since (a) of Lemma \ref{contorn} says that $|u(z)|\le \delta$ when $\beta(z,\interi\Gamma) = 1$, and we can assume that $2C_2 \ge 1$,
$$
\log \frac{1}{\delta} \le 2C_2 \left[ \log \frac{1}{|u_2(z)|} +  \log \frac{1}{|b_1(z)|} \right] +K .
$$
Therefore
$$
|u_2(z)b_1(z)| \le (e^{K} \delta)^{1/2C_2}  = c(\delta),
$$
which together with the maximum modulus principle proves \eqref{apumante}.
Summing up, if $\delta$ is sufficiently small, we
can apply Propositions \ref{modone2} and \ref{modone3} to deduce
that there exists $h \in (\papa)^{-1}$ such that $u$ and $u_2 b_1
h$ can be joined by a path contained in $\inni$.
\end{proof}

\begin{rema}\laba{careful}
{\em
A careful examination of the proofs of Proposition \ref{modone2} and Theorem \ref{components} shows that there exists a
universal constant $N$ such that any inner function
can be joined in $\inni$ to a function in $\cni$ by a polygonal formed by the union of at most $N$ segments.}
\end{rema}

\section{Applications and examples }

\subsection{The  invertible group of a Douglas algebra}

Given an inner function $u$, the Douglas algebra $\papa[\ov{u}]$ is
the closed subalgebra of $L^\infty (\partial \disc)$ generated by $\papa$ and $\ov{u}$.
The maximal ideal space of $\papa[\ov{u}]$ is naturally identified with
the subset of the maximal ideal space $\chu $  of $\papa$ given by
$$
M_u  = \{ x\in \chu : |u(x)| =1 \} .
$$
Here we are looking at the functions of $\papa$ as defined on the whole maximal space $\chu$
(that is, we are identifying $f\in \papa$ with its Gelfand transform).
Given two inner functions $u_0$ and $u_1$, it is well known that
$\papa[\ov{u}_0]   \subset \papa[\ov{u}_1]$ if and only if $M_{u_0} \supset M_{u_1}$ (see \cite[IX]{gar}).

\begin{lemma}\laba{mofone}
For $0 \leq t \leq 1$ let $u_t$ be an inner function such that the
mapping $ t \rightarrow |u_t|$ is continuous from $[0,1]$ to
$L^\infty (\disc)$. Then, given $\eps >0$ there is $\delta >0$ such that
$$
|u_0(z)| >1-\delta  \ \Rightarrow\  |u_t(z)| >1-\eps \ \mbox{ for all\/  $\,0\le t\le 1$}.
$$
In particular, $\papa[\ov{u}_0]=\papa[\ov{u}_t]$ for all $\,0\le t\le 1$.
\end{lemma}
\bdem For any $t_0\in [0,1]$ we have $|u_t (z)|-1/8 \leq |u_{t_0}
(z)| \leq |u_t (z)|+1/8$ for all $z \in \disc$ whenever $|t-t_0|$ is
small enough. Therefore, for these values of $t$ one has
$$
\{ z\in \disc : |u_t(z)| > \frac{1}{2} \}     \subset \{ z\in \disc
: |u_{t_0}(z)| > \frac{3}{8} \} \subset \{ z\in \disc : |u_t(z)| >
\frac{1}{4} \} .
$$
The first inclusion implies that $u_{t_0}^{-1} \in \papa(\{ |u_t| >
1/2 \})$, and since $u_{t_0}^{-1}(\eiti) = \ov{u}_{t_0}(\eiti)$ for
almost every $\theta$,  \cite[IX, Thm.$\,$5.2]{gar} says that
$\ov{u}_{t_0} \in \papa[\ov{u}_t]$. Analogously, the second
inclusion shows that $\ov{u}_t \in \papa[\ov{u}_{t_0}]$. That is,
$\papa[\ov{u}_t]=\papa[\ov{u}_{t_0}]$.
Furthermore,
let $\delta_n >0$ be a sequence that tends to $0$ and denote $I=[0,1]$. The set
$$
\{ (t,x)\in I\times\chu : |u_t(x)| >1-\eps \}
$$
is an open neighborhood of $I\times M_{u_0}=\bigcap_n \{ (t,x)\in
I\times \chu : |u_0(x)| > 1-\delta_n \}$. So, by compactness there is
some $n$ such that
$$
\{ (t,x) : |u_0(x)| > 1-\delta_n \} \subset \{ (t,x) : |u_t(x)|
>1-\eps \} .
$$
\edem

\noi An immediate corollary of Lemma \ref{mofone} is that if $u_t$
$(0\le t\le 1)$  are inner functions such that $|u_t|$ varies continuously in $\| \ \|_\infty$ and $u_0$ is a floating Blaschke
product, then
$$
\inf_\theta |u_0(r_n \eiti)|\rr 1  \ \Rightarrow\ \inf_\theta
|u_t(r_n \eiti)|\rr 1 \ \mbox{ uniformly on\/  $\,0\le t\le 1$}.
$$
In particular, $u_t$ is a floating Blaschke product for all $t$.
Observe that this argument shows that in Propositios \ref{modone2} and \ref{modone3}, the inner function $u_1$ is also a floating Blaschke product.

If $A$ is a commutative Banach algebra with unit, and $A^{-1}$ is the group of invertible elements, the connected component of the unit in  $A^{-1}$
is $\exp A = \{ e^a : a\in A \}$. Therefore, two elements $a, b\in A^{-1}$ are in the same component if and only if $b \in a \exp A$.

\begin{theo}\laba{modone}
Let $u_0, u_1$ be two inner functions such that $M_{u_0} = M_{u_1}$ and $h_1 \in (\papa)^{-1}$.
The following conditions are equivalent
\begin{enumerate}
\item[{\em (1)}] $\ h_1 u_1 \in u_0 \exp(\papa[\ov{u}_0])$.
\item[{\em (2)}] for some open neighborhood $\,U$ of $M_{u_0}$ in $\chu$ with $\inf_U |u_0 \, u_1|  >0$,
there is a bounded analytic branch of\/ $\,\log ( h_1u_1/ u_0)$ on $U\cap \disc$.
\item[{\em (3)}] There exists a \CNBP $\ b$ {\em (}or $b\equiv 1${\em )}
with $M_b\supseteq M_{u_0}$ such that $bu_0$ and $bh_1 u_1$ can be joined by a path contained in $\inni$.
\end{enumerate}
\end{theo}
\begin{proof}

\noi $(1) \Rightarrow (2)$. Let $f\in \papa[\ov{u}_0]$ such that $h_1 u_1 = u_0 e^f$.
Since $f\in\papa[\ov{u}_0]$, given $\eta >0$
there are $g\in \papa$ and $n\geq 0$ integer such that
$\sup_{M_{u_0}} |\ov{u}_0^n g + f| < \eta$. The set $\{ x\in \chu :
|u_0(x)| >1/2 \}$ is a neighborhood of $M_{u_0}$ where the function
$$
\Lambda_\eta:= \left|\frac{u_1h_1}{u_0} \exp( \frac{g}{u_0^n} )
-1 \right|
$$
is continuous. In addition, on $M_{u_0}$:
$\Lambda_\eta  =  | e^{f+\ov{u}_0^n g} -1|  \le   e^\eta-1$.
Thus, by choosing $\eta>0$ small enough so
that $\sup_{M_{u_0}} \Lambda_\eta < 1/4$, we get that
$$U:=\{ x\in \chu : |u_0(x)| >\frac{1}{2} \, \mbox{ and }\, |\Lambda_\eta(x)| < \frac{1}{2} \}
$$
is an open neighborhood of $M_{u_0}$ such that the function $(u_1h_1 /u_0) \exp( g /u_0^n )$ has
a bounded analytic logarithm on $U\cap\disc$. Clearly,
so does $ u_1 h_1 / u_0$.

\noi $(2) \Rightarrow (1)$. Let $q\in \papa(U\cap\disc)$ such that $h_1u_1/u_0 = e^q$ on $U\cap\disc$, where $U$ is an open
neighborhood of $M_{u_0}$. Then $U\cap\disc$ is  non-tangentially dense,
and the
function $q$ has a non-tangential limit at almost every point of
$\partial\disc$ that belongs to $\papa[\ov{u}_0]$ (see \cite[IX, Thm.$\,$5.1 and Thm.$\,$5.2]{gar}).

\noi $(1) \Rightarrow (3)$. Let $f\in \papa[\ov{u}_0]$ such that $h_1 u_1 =  u_0 e^f$. Since the set
$$
\left\{ g\frac{a}{b}: \ g\in (\papa)^{-1}, \ a, b\in \mbox{CNBP} \cap
\papa[\ov{u}_0]^{-1} \right\}
$$
is dense in $\papa[\ov{u}_0]^{-1}$ (see \cite[Thm.$\,$3.3]{sua3}), the homotopy $e^{tf}$,
$0 \leq t \leq 1$,  in $\papa[\ov{u}_0]^{-1}$ can be approximated by
a polygonal $p(t)$ formed by segments joining finitely many
functions of the form
$$
p(0)=1,\, g_0\frac{a_0}{b_0}, \,  g_1\frac{a_1}{b_1}, \ldots , \, g_n\frac{a_n}{b_n}, \, e^f = p(1),
$$
where $g_j \in (\papa)^{-1}$, and   $a_j, b_j \in \mbox{\CNBP} \cap \papa[\ov{u}_0]^{-1}$. Setting $b= \prod_{j=0}^n b_j$, we have that
$bu_0 p(t)$, $0 \leq t \leq 1$, implements a path in $\inni$
between $bu_0$ and $bu_0  e^f= bh_1u_1$.
Since each $b_j \in\papa[\ov{u}_0]^{-1}$, so is $b$, meaning that $M_b \supset M_{u_0}$.

\noi $(3) \Rightarrow (1)$.
Let $\gamma : [0,1] \to \inni$ be a path joining $\gamma (0)= bu_0$ with $\gamma(1)= bh_1u_1$, where $b$ is as in (3),
and denote by $v_t$ the inner part of $\gamma(t)$.
By Lemma \ref{moddy}, the map $t\mapsto |v_t|$ is continuous from $[0,1]$ into $L^\infty(\disc)$, and
consequently Lemma \ref{mofone} says that $\gamma(t) \in (\papa[\ov{v}_0])^{-1}$ for all $t\in [0,1]$.
Hence, $bu_0$ and $bh_1 u_1$ are in the same connected component of $(\papa[\ov{bu}_0])^{-1}$, meaning that
$$\ bh_1 u_1 \in b u_0 \exp(\papa[\ov{bu}_0])= b u_0 \exp(\papa[\ov{u}_0]),$$
where the last equality holds because the hypothesis $M_b\supset M_{u_0}$ implies that
$\papa[\ov{bu}_0]=\papa[\ov{u}_0]$.
So, multiplying the above formula by $\ov{b}$ we obtain (1).
\end{proof}

\subsection{Nice Blaschke products}

For a Blaschke product $b$ and $r>0$ write
$$
\alpha_b(r)= \inf \{ |b(z)| : \ \beta(z,Z(b)) >r \} .
$$
This function increases with $r$, so $\alpha_b(\infty):= \sup_r \alpha_b(r)= \lim_{r\rr \infty} \alpha_b(r)\in [0,1]$.
It is well known that if $b$ is a \CNBP\ then $\alpha_b(\infty) >0$. For a while
it was mistakingly believed that the converse is also true.
However, in \cite{g-m-n} the authors exhibit a Blaschke product $b$ constructed by Treil that satisfies
$\alpha_b(\infty) >0$ but it is not a \CNBP. Moreover, a quick examination of the example shows that
$\alpha_b(\infty) =1$.
Notice that $\alpha_b(\infty)=1$ just means that $|b(z)| \rr 1$ when $\beta(z, Z(b)) \rr \infty$.

The significance of this constant is that if $w\in\disc$ satisfies $0<|w| < \alpha_b(\infty)$ then
$b_w  = (w -b)/(1-\ov{w}b)$ is a \CNBP. Indeed, if $r>0$ is such that
$|w| < \alpha_b(r)$ then for every $z\in Z(b_w)$ there is some point $\xi\in Z(b)$ with $\beta(z,\xi) \le r$.
So, $|b_w(\xi)| = |w|$, implying that $b_w$ is a \CNBP.

Let $\Gamma_{\papa}$ be the set of trivial points in
$M(\papa)$, that is, the points of $\chu$ whose Gleason part is a
singleton. It is well known that an inner function is a \CNBP $\ $if and only if it never vanishes on $\Gamma_{\papa}$.
\begin{coro}\laba{alafa}
Let $b$ be a \CNBP. Then
$$\alpha_b(\infty) = \inf \{ |b(x)| : x\in \Gamma_{\papa}\}.$$
\end{coro}
\begin{proof}
Let us denote the above infimum by $\gamma$.
If $|w| < \alpha_b(\infty)$ then $b_w  = (w -b)/(1-\ov{w}b)$ is a Carleson-Newman Blaschke product, and consequently
never vanishes on $\Gamma_{\papa}$. So, $\gamma \ge \alpha_b(\infty)$.

If $\gamma >\alpha_b(\infty)$ there is a sequence $\{ z_n \}$ such that $\beta( z_n ,Z(b)) \rr \infty$ and $b(z_n) \rr \lambda$,
with $\alpha_b(\infty)<|\lambda |<\gamma$. The last of these inequalities implies that
$b_\lambda  = (\lambda -b)/(1-\ov{\lambda}b)$ is a \CNBP, and consequently $\beta(z_n, Z(b_\lambda)) \rr 0$ when
$n\rr \infty$. Hence, there is a subsequence of zeros of $b_\lambda$, say $\{ w_k \}$, such that $\beta(w_k, Z(b))\rr \infty$.
By Theorem \ref{cra}, $b$ and $b_\lambda$ cannot be joined by a continuous path contained in $\cni$,
which means the path $b_{t\lambda}$, with $0\leq t\leq 1$, cannot consist entirely of \CNBP.
In other words, there is $t_0$, $0 \leq t_0 \leq 1$ such that $b_{t_0\lambda}$ vanishes at some point of $\Gamma_{\papa}$,
a contradiction.
\end{proof}

\noi
A description of the \CNBP s $b$ that satisfy $\alpha_b(\infty)=1$ in terms of the distribution
of their zeros can be found in \cite{nico}.
The techniques are based on a previous result by Bishop \cite{bis2}, where he characterized the Blaschke products
in the little Bloch space $\mathcal{B}_0$ in terms of their zeros. Not surprisingly, the distribution of the zeros
in both cases are diametrically opposed.
Similarly, we have seen in  Theorem \ref{litbloch} the bad behaviour of the Blaschke products in $\mathcal{B}_0$ with respect
to the components of $\cni$, and next we show how nicely behaves a \CNBP $\ b$ with $\alpha_b(\infty)= 1$.

\begin{coro}\laba{buenos}
Let $b$ be a \CNBP$\ $and $h\in (\papa)^{-1}$. Then $\alpha_b(\infty)=1$ if and only if
$\cni(hb)= \inni(hb)$, where $\cni(hb)$ (respectively $\inni(hb)$) is the component  of $hb$ in $\cni$ (respectively $\inni$).
\end{coro}
\bdem
First assume that $\alpha_b(\infty)=1$.
We prove the nontrivial inclusion. So, suppose that $u_t h_t\in \inni$ is a path,
where $u_t$ is inner, $h_t\in (\papa)^{-1}$ and $u_0 h_0 = b h$. By Corollary \ref{alafa} and Lemma \ref{mofone},
$$
\Gamma_{\papa} \subset \{ x\in \chu : |b(x)|=1 \} = \{ x\in \chu : |u_t(x)|=1 \}
$$
for all $t$. In particular, $u_t$ never vanishes on $\Gamma_{\papa}$ and therefore it is a \CNBP.
Thus, the path is actually in $\cni$.
Now suppose that $\alpha_b(\infty)<1$. Then by Corollary \ref{alafa} there is some $w\in\disc$ such that $b_w= (w-b)/ (1-\ov{w}b)$ vanishes at some point
of $\Gamma_{\papa}$. Thus, $hb_w \in \inni(hb)\setminus \cni(hb)$.
\edem

\noi
In \cite{nes2} Nestoridis proved that if $u$ is an inner function such that for every $0<\varepsilon <1$,
the hyperbolic diameter of the components of
$\{ z\in\disc: \, |u(z)|< \varepsilon \}$ is bounded by a constant depending on $\varepsilon$,
then $u$ and $zu$ cannot be joined by a path of inner functions.
Since it is clear that such $u$ must be a \CNBP $\ $satisfying $\alpha_u(\infty)=1$,
Corollaries \ref{buenos} and  \ref{ultim} imply that
there is no $h \in (\papa)^{-1}$ such that $uh$ and $zuh$ are in the same component of $\inni$.

\subsection{Oddities}

\begin{propos}\laba{cami}
Let $f,g\in \inni$. Then there is a \CNBP $\  b$ such that $bf$ and $bg$ can be joined by a path contained in\/
$\inni$. Moreover, if $f,g\in \cni$, then $b$ can be chosen such
that\/ $bf$ and $bg$ are joined by a path contained in $\cni$.
\end{propos}
\bdem As will be explained later, only the second statement needs to be proved. So let $f,g\in \cni$.
It is known that $\Gamma_{\papa}$ is totally disconnected (see
\cite[Thm.$\,$3.4]{sua2}), and that the set of functions
$$
\{h  b_{1}/b_{2}  : h\in (\papa )^{-1} \mbox{ and } b_{1}, b_{2}
\mbox{ are \CNBP } \}
$$
is dense in $C(\Gamma_{\papa})$ (see the comments preceding Lemma 4.3 in \cite{sua2}).
Since $\Gamma_{\papa}$ is totally disconnected, $C(\Gamma_{\papa})^{-1}$ is
connected (see \cite[Thm.$\,$III.4]{nag}), and consequently there is a path in
$C(\Gamma_{\papa})^{-1}$ joining $f$ with $g$. We can assume that this path
is a polygonal $\gamma : [0,1] \rightarrow C(\Gamma_{\papa})^{-1}$ joining
finitely many functions:
$$
f,\, h_0\frac{a_0}{b_0}, \,  h_1\frac{a_1}{b_1}, \ldots , \, h_n\frac{a_n}{b_n},\, g
$$
where
$h_j \in (\papa)^{-1}$ and $a_j, b_j$ are \CNBP.
Consider $b= \prod_{j=0}^n b_j$, then $b \gamma(t)$ is a polygonal
in $\cni$ that joins $bf$ with $bg$. The proof of the first
statement is analogous once $\Gamma_{\papa}$ is  replaced by the Shilov
boundary of $\papa$.
\edem

Let $h \in (\papa)^{-1}$ which is not in the connected component of the unity. By Proposition \ref{cami} there exists a \CNBP$\ b$ such that
$b$ and $bh$ are in the same component of $\cni$. Let $b_t h_t$, $0 \le t \le 1$, be the path joining $b$ and $bh$ in $\cni$. Observe that $t \mapsto b_t$ cannot be continuous because $t \mapsto h_t$ is not.

In \cite{jones1} Jones used interpolating Blaschke products to find a
constructive method of obtaining the Fefferman-Stein decomposition
of a BMO function. Our next result points in the same direction.

\begin{coro} Let $u$ be a real-valued function in $L^\infty(\partial \disc)$, and let $\tilde{u}$ be its harmonic
conjugate. Then there exists a \CNBP $\ b$ with zeros $\{ z_k \}$ and a permutation of these zeros
$\{ z^\ast_k \}$ with $\sup \beta (z_k,z^\ast_k) < \infty$ such that
if $\sigma$ is the measure associated to these zeros by the comments
preceding Lemma \ref{intwin}, then {\em
$$
\tilde{u} \in \Im \ca(\sigma) + L^\infty(\partial \disc) .
$$
}
\end{coro}
\bdem Consider the function $h= e^{ u+i\tilde{u} } \in (\papa)^{-1}$
and apply the Proposition \ref{cami} to $h$ and $1$. Then there is a \CNBP$\ b$ such that $bh$ and $b$ can be
joined by a path contained in $\cni$. Applying Theorem \ref{cra} the
proof is completed. \edem

\subsection{Open questions}

This subsection is devoted to mention several questions that appear naturally
in this context.

The first one is to find a description of the connected components
of $\inni$ from which one could deduce our main result,
Theorem \ref{components}. Let $u$ and $b$ be inner functions and $h \in \papai$.
According to Lemma \ref{mofone} and Theorem \ref{modone} if $u$ and $bh$ are in the same component of
$\inni$ then $M_u = M_b$ and there is a bounded branch of the
logarithm of $u/bh$ in a natural open subset of the unit disk.
However, these conditions do not seem to be sufficient.

Let $u$ be an inner function and let $\inni (u)$ be the connected
component of $\inni$ containing $u$. Does there exist a function $f$
in $\cni$ such that the segment $\{u + t(f-u) : 0 \leq
t \leq 1 \}$ is contained in $\inni (u)$? In other words, can we take $N=1$ in Remark \ref{careful}?

Our main result says that any inner function can be joined in
$\inni$ to a function in $\cni$. Can it be also joined in $\inni$ to a \CNBP? Observe that a
positive answer to Problem \ref{prob1} would
imply a positive answer to this question. Indeed,  if $\|u - b
\|_\infty <1$, the segment $\{u + t(b-u) : 0 \leq t \leq 1 \}$ is
contained in $\inni$. Also, it is not difficult to show that there are
plenty of components of $\inni\setminus (\papa)^{-1}$
which contain no inner function. An easy example is provided by a finite Blaschke product $b$
and $h\in (\papa)^{-1}\setminus\exp \papa$. It follows immediately from Theorem \ref{cra} that there is no inner
function in $\inni(bh)$.

Theorem \ref{litbloch} tells us that the boundary of a single connected component of
$\cni$ cannot contain any inner function in the little Bloch space
except for finite Blaschke products. It is natural to ask for a
description of the functions which are in the boundary of a
connected component of $\cni$.

Given a component $U$ of $\,\inni\!$, describe all the components of $\cni$ contained in $U$. Corollary \ref{buenos} gives an answer in a particular case.
\vspace{5mm}

\noindent {\bf Acknowledgements:}
Both authors are supported in part by the grants
MTM2009-00145 and 2009SGR420.
The second author has also being supported by the program Ram\'on y Cajal
while expending several years in the wonderful environment of the Universitat Aut\`{o}noma de Barcelona.
\hyphenation{Cajal}

\bibliographystyle{amsplain}

 \newcommand{\foo}{\footnotesize}

 \bigskip

 \noindent Artur Nicolau\\
 Departament de Matem\`{a}tiques \\
  Universitat Aut\`{o}noma de Barcelona \\
 08193, Bellaterra, Barcelona \\
 Spain\\
 \vspace{0.5mm} \noindent $\! \!${\foo artur@mat.uab.es}
  \mbox{  }
 \vspace{0.1cm} \mbox{ }

\noindent Daniel Su\'{a}rez\\
 Departamento de Matem\'{a}tica \\
 Facultad de Ciencias Exactas y Naturales \\
 UBA, Pab. I, Ciudad Universitaria \\
 (1428) N\'{u}\~{n}ez, Capital Federal \\
 Argentina\\
\vspace{0.5mm} \noindent $\! \!${\foo dsuarez@dm.uba.ar}

\end{document}